\documentclass[oneside,english]{amsart}
\usepackage{ae,aecompl}
\usepackage[T1]{fontenc}
\usepackage[latin9]{inputenc}
\usepackage{geometry}
\usepackage{amstext}
\usepackage{amsthm}
\usepackage{amssymb}
\usepackage{setspace}
\usepackage[authoryear]{natbib}
\makeatletter
\numberwithin{equation}{section}
\numberwithin{figure}{section}
\theoremstyle{plain}
\newtheorem{thm}{\protect\theoremname}
\theoremstyle{plain}
\newtheorem{prop}[thm]{\protect\propositionname}
\theoremstyle{plain}
\newtheorem{fact}[thm]{\protect\factname}
\theoremstyle{definition}
\newtheorem{defn}[thm]{\protect\definitionname}
\theoremstyle{plain}
\newtheorem{lem}[thm]{\protect\lemmaname}
\theoremstyle{remark}
\newtheorem*{rem*}{\protect\remarkname}
\theoremstyle{plain}
\newtheorem{cor}[thm]{\protect\corollaryname}

\makeatother

\usepackage{babel}
\providecommand{\corollaryname}{Corollary}
\providecommand{\definitionname}{Definition}
\providecommand{\factname}{Fact}
\providecommand{\lemmaname}{Lemma}
\providecommand{\propositionname}{Proposition}
\providecommand{\remarkname}{Remark}
\providecommand{\theoremname}{Theorem}

\begin{document}
\title{Teasing definitional equivalence and bi-interpretability apart}
\author{Jason Chen \& Toby Meadows}
\begin{abstract}
In a recent paper, \citet{EnayatLelykFOCat} show that second order
arithmetic and countable set theory are not definitionally equivalent.
It is well known that these theories are bi-interpretable. Thus, we
have a pair of natural theories that illustrate a meaningful difference
between definitional equivalence and bi-interpretability. This is
particularly interesting given that \citet{VisserFriedBitoSyn} have
shown that a wide class of natural foundational theories in mathematics
are such that if they are bi-interpretable, then they are also definitionally
equivalent. The proof offered by \citeauthor{EnayatLelykFOCat} makes
use of an inaccessible cardinal. In this short note, we show that
the failure of bi-interpretability can be established in Peano Arithmetic
merely supposing that one of our target theories are consistent. 
\end{abstract}

\thanks{We're grateful to Ali Enayat for looking over our sketches of these
proofs and providing helpful insights and encouragement.}

\maketitle
We begin by recalling some basic definitions and set up our notation.
A more precise and detailed discussion can be found in \citep{Visser2006,VisserFriedBitoSyn,HalvLogPhilSci,ButtonWalsh}
or \citep{meadowsBLI}. Let $T$ and $S$ be theories articulated
in the language $\mathcal{L}_{T}$ and $\mathcal{L}_{S}$ respectively.
Suppose that they are mutually interpretable as witnessed by translations
$t:\mathcal{L}_{S}\to\mathcal{L}_{T}$ and $s:\mathcal{L}_{T}\to\mathcal{L}_{S}$
giving rise to functions $t:mod(T)\to mod(S)$ and $s:mod(S)\to mod(T)$
where $mod(T)$ and $mod(S)$ are the classes of models satisfying
$T$ and $S$ respectively.\footnote{We permit translations that make use of multiple dimensions and quotients. }
We say that $T$ and $S$ are \emph{definitionally} \emph{equivalent}
if:
\begin{enumerate}
\item $\mathcal{A}=s\circ t(\mathcal{A})$ for all models $\mathcal{A}$
of $T$; and
\item $\mathcal{B}=t\circ s(\mathcal{B})$ for all models $\mathcal{B}$
of $S$. 
\end{enumerate}
On the other hand, we say that $T$ and $S$ are \emph{bi-interpretable
}if, in addition to $t$ and $s$ witnessing mutual interpretability,
there are functions $\eta$ and $\nu$, uniformly definable over $T$
and $S$ respectively, such that:
\begin{enumerate}
\item $\eta^{\mathcal{A}}:\mathcal{A}\cong s\circ t(\mathcal{A})$ for all
models $\mathcal{A}$ of $T$; and
\item $\nu^{\mathcal{B}}:\mathcal{B}\cong t\circ s(\mathcal{B})$ for all
models $\mathcal{B}$ of $S$. 
\end{enumerate}
Informally, we have definitional equivalence when we have translations
that allows us to go back and forth to exactly where we started. Bi-interpretability,
by contrast, is weaker in that we only return to an isomorphic structure
where the relevant isomorphism is definable. For most mathematical
purposes bi-interpretability seems to be compelling evidence that
the theories in question can be regarded as informally equivalent. 

However, it is not difficult to find a toy example that pulls these
equivalence relations apart. Let $\mathcal{L}_{T}$ be the empty language
and let $T$ be the theory that says there are infinitely many objects.
More specifically, we let $T$ consist of the sentences $\exists_{\geq n}x\ x=x$
for all $n\in\omega$. Let $\mathcal{L}_{S}$ be the language consisting
of a single constant symbol $c$. And let $S$ consist of the same
sentences as $T$. Thus, $S$ says nothing at all about $c$. We now
sketch a quick proof of the following:
\begin{prop}
$T$ and $S$ are bi-interpretable, but not definitionally equivalent.\label{prop:simple}
\end{prop}

\begin{proof}
(Bi-interpretability) The essential idea is to move from a model of
$S$ to a model of $T$ by discarding the object denoted by $c$;
and in the other direction, we move from a model of $T$ to models
of $S$ by using a three-dimensional quotient interpretation $t$
to ``add'' a new object that will can be denoted by $c$. More specifically,
given a model $\mathcal{A}=\langle A\rangle$ of $T$, we let our
domain be $A^{3}$ and then define an interpretation for $\dot{=}$
and $c$. Given $\bar{x}=\langle x_{1},x_{2},x_{3}\rangle$, $\bar{y}=\langle y_{1},y_{2},y_{3}\rangle$
from $A^{3}$ we let 
\[
\bar{x}\dot{=}\bar{y}\ \Leftrightarrow\ (x_{2}=x_{3}\wedge y_{2}=y_{3}\wedge x_{1}=y_{1})\vee(x_{2}\neq x_{3}\wedge y_{2}\neq y_{3})
\]
and
\[
\bar{x}=c\ \Leftrightarrow\ x_{2}\neq x_{3}.
\]

Putting all of this together, we let $t(\mathcal{A})=\langle A^{3},c^{\mathcal{A}},\dot{=}\rangle$.
In the other direction, we start with a model $\mathcal{B}=\langle B,b\rangle$
of $S$ and provide an interpretation $s$ by defining a new domain
using the formula 
\[
\delta_{s}(x):=(x\neq c).
\]
We then let $s(\mathcal{B})=\langle\delta_{s}^{\mathcal{B}}\rangle$
where $\delta_{s}^{\mathcal{B}}=\{x\in B\ |\ \mathcal{B}\models\delta_{s}(x)\}$.
Clearly, $t$ and $s$ witness mutual interpretability. Finally, we
define the required isomorphisms. Given a model $\mathcal{A}$ of
$T$, we have $s\circ t(\mathcal{A})=\langle E,\dot{=}^{s\circ t(\mathcal{A})}\rangle$
where $E=\{\bar{x}\in A^{3}\ |\ \mathcal{A}\models x_{2}=x_{3}\}$
and for $\bar{x},\bar{y}\in E$, $\dot{=}^{s\circ t(\mathcal{A})}=\{\langle\bar{x},\bar{y}\rangle\in(A^{3})^{2}\ |\ \mathcal{A}\models x_{1}=y_{1}\}$.
Thus, we may let our isomorphism be defined by the formula $\eta(x,\bar{y})$
be $x=y_{1}=y_{2}=y_{3}$. The definition of $\nu$ uses similar techniques
and we leave it to the reader. 

(Not definitionally equivalence) Suppose toward a contradiction we
have interpretation as giving rise to functors $t:mod(T)\to mod(S)$
and $s:mod(S)\to mod(T)$ witnessing the definitional equivalence
of $T$ and $S$. If we start with a model $\mathcal{B}=\langle B,b\rangle$
of $S$, we have just two choices for our interpretation: 
\begin{enumerate}
\item we can remove -- as we did above -- $b$ from the domain to obtain
$\langle B\backslash b\rangle$; or 
\item we can simply forget the denotation of $b$ and retain the domain
to give $\langle B\rangle$. 
\end{enumerate}
If we take option (1), we cannot have definitional equivalence since
we have moved to a proper subset of the original domain, which cannot
be recovered. So we are stuck with option (2) and we have $t(\mathcal{B})=\langle B\rangle$.
Now if we are to have $s\circ t(\mathcal{B})=\mathcal{B}$, we must
have $s(\langle B\rangle)=\langle B,x\rangle$ for some $x\in B$.
But this would mean that we could define an element of an infinite
set in the empty language. This is plainly impossible as can be established
with a simple automorphism argument.
\end{proof}
So much for toy theories. If we move to theories with more serious
foundational credentials, things get more interesting. First we recall
a few definitions.\footnote{These can be found in \citep{VisserFriedBitoSyn} but we include them
to make things a little more self-contained. } We say that an interpretation is \emph{one-dimensional}, if the formula
defining the domain of the interpretation has one free variable and
thus, defines a subset of domain of models of the interpreting theory.
Say that an interpretation is \emph{identity preserving }if it translates
the identity predicate to itself.\footnote{So the translation $t$ in the proof of Proposition \ref{prop:simple}
is multi-dimensional since it uses ordered triples. Moreover, we see
that it is not identity-preserving since, in order to form the quotient,
we make an explicit definition of the identity relation, denoted $\dot{=}$,
on those triples.} We say that a one-dimensional interpretation is \emph{unrelativized
}if it does not restrict the domain.\footnote{The interpretation $s$ in Proposition \ref{prop:simple} is relativized
since we use the formula $\delta_{s}$ to remove $c$ from the domain
of the interpretation.} We say that a one-dimensional interpretation is \emph{direct} if
it is unrelativized and preserves identity. A theory is said to be
\emph{sequential }if it directly interprets adjunctive set theory,
which is the theory saying that: there is an empty set; and for any
sets $x,y$ there is some $z$ containing exactly $y$ and the members
of $x$. Such a theory is capable of doing some basic coding. Any
serious foundational theory is obviously sequential. 
\begin{thm}
\citep{VisserFriedBitoSyn} Let $T$ be a sequential theory.\footnote{\citeauthor{VisserFriedBitoSyn} actually use a weaker class of theories
called \emph{conceptual theories} in their paper, but this will not
affect the discussion in this paper.} Suppose that $T$ and $S$ are bi-interpretable as witnessed by one-dimensional,
identity preserving interpretations. Then $T$ and $S$ are definitionally
equivalent. 
\end{thm}

This theorem is particularly helpful. For example, it is well-known
that: Peano arithmetic is bi-interpretable with a finite version of
$ZFC$;\footnote{Some care is required in the axiomatization. We assume that we are
using set induction rather than foundation. Or we can add an axiom
stating that every set is contained in a transitive set. See \citep{KayeSTinArith}.} and $ZFC$ is bi-interpretable with $ZFC$ with foundation removed
and replaced by Aczel's anti-foundation axiom. Each of these theories
is sequential and the interpretations linking them can be arranged
to be identity preserving. Thus, we see that both pairs are actually
examples of definitionally equivalent theories. However, a crucial
element in these argument is the ability of these theories to eliminate
the use of quotient interpretations by finding representatives for
the equivalence classes. In arithmetic, we just need to pick the least
element; in set theory, we use Scott's trick, whereby we take the
set of elements of the equivalence class of minimal rank. But not
all theories can perform this kind of trick. In set theory, we seem
to need some form of reflection in order to eliminate quotients.\footnote{In model theory, this is known as \emph{eliminating imaginaries}.
See Section 4.4 of \citep{Hodges}.} 

This idea provides a lead toward a natural pair of theories that are
bi-interpretable but not definitionally equivalent.\footnote{We note that \citet{VisserFriedBitoSyn} do provide a pair of sequential
theories that are bi-interpretable but not definitionally equivalent.
The example is interesting, however, the theories are not in common
use and might be thought of as being contrived for the purposes of
the result. As such, we regard them as an \emph{unnatural} pair of
theories. } In particular, $ZFC^{-}$ ($ZFC$ without the powerset axiom)\footnote{We should also use Collection rather than Replacement. See \citep{GitmanZFC-}.}
cannot perform Scott's trick.\footnote{It is worth noting that $ZFC^{-}$ can be augmented to a theory that
can eliminate imaginaries by, for example, adding the assumption that
$V=L$ which ensures that the universe has a definable well-ordering.} And indeed, it is in this area that \citeauthor{EnayatLelykFOCat}
find their example. Let $ZFC_{count}$ be $ZFC^{-}$ with an axiom
stating that every set is countable.\footnote{Note that the $C$ in $ZFC_{count}$ is redundant, since the $_{count}$
axiom ensures that every set is well-ordered by a bijection with $\omega$.} Let $SOA$ be the theory of second order arithmetic with full comprehension
and choice for all definable sets of reals indexed by naturals.\footnote{We adopt the logicians' reals in the paper and say that $\mathbb{R}=\mathcal{P}(\omega)$.
The choice schema then says that for any formula $\varphi(n,Y)$ in
the language of $SOA$, if for all $n$ there is some $Y$ such that
$\varphi(n,Y)$, then there is some $Z$ such that for all $n$, $\varphi(n,(Z)_{n})$,
where $(Z_{n})=\{i\in\omega\ |\ (i,n)\in Z\}$. See \citep{EnayatLelykFOCat}
and \citep{Simpson} for more details. In the latter, this theory
is denoted as $\Sigma_{\infty}^{1}-AC_{0}$ noting that the definable
choice principles end up implying the comprehension principles. See
Section VII.6 and Lemma VII.6.6 in \citep{Simpson} for more information.
We note that these choice principles are required since there are
models of second order arithmetic with full comprehension where definable
choice fails for a $\Sigma_{3}^{1}$-set: see Remarks VII.6.3 in \citep{Simpson}
and Example 15.57 in \citep{JechST}. Finally, we also note that $SOA$
is also often notated as $Z_{2}$ while $PA$ is often denoted as
$Z_{1}$ \citep{HilbertGrundlagen}. However, it is not always clear
whether $Z_{2}$ is intended to include the definable choice principles
we incorporate in $SOA$.} First we note that:
\begin{thm}
\citep{MostSOAZFCcountOP}\footnote{The attribution for this result is a little convoluted. One reason
for this is that while the bones of the proof seem to have been around
since the late 1950s, the definition of bi-interpretability wasn't
formally isolated until \citep{AHLBbi-int}, although similar ideas
were in circulation in the 1970s: see, for example, \citep{OSIUS197479}.
We follow convention and attribute the result to Mostowski, although
\citep{MostSOAZFCcountOP} only appears to contain the easy direction
of the proof that delivers a model of $SOA$ from $ZFC_{count}$ essentially
by truncation: see Theorem 7.15 in \citep{MostSOAZFCcount}. Given
that the key trick in the other, more difficult direction involves
Mostowski's famous collapse function, this convention still seems
apropos. A detailed proof that $SOA$ can interpret $ZFC_{count}$
using trees is provided in Theorem 5.5 of \citep{AptSOA} where they
attribute the result to \citep{Kreisel_SOA} and \citep{ZbierskiHOA}.
More subtle results on these interpretations can be found in Section
VII.3 of \citep{Simpson} and, more recently, \citep{kanovei2025notesequiconsistencyzfcpower}. } $ZFC_{count}$ is bi-interpretable with $SOA$. 
\end{thm}

Just to give the idea, we first note that since $PA$ and finite $ZFC$
are definitionally equivalent,\footnote{This follows from Corollary 5.5 in \citep{VisserFriedBitoSyn} and
the fact that $PA$ and finite $ZFC$ are bi-interpretable via identity
preserving interpretations. Note also that by finite $ZFC$, we mean
a theory that uses the Set Induction schema rather than Foundation;
or alternatively, we could include an axiom saying that every set
is contained in a transitive set. See \citep{KayeSTinArith} for more
details. Also note that we could just use $\omega$ and its subsets
from a model of $ZFC_{count}$ to deliver a model of $SOA$.} we can, without harm, think of the number domain of a model of $SOA$
as being $V_{\omega}$. To move from a model of $ZFC_{count}$ to
a model of $SOA$ we just forget all the sets with rank $>\omega$.
This direction is trivial. In the other direction, we start with a
model of $SOA$ and simulate the effect of hereditarily countable
sets using well-founded, extensional relations $R$ on $\omega$ with
top elements that collapse to be such sets.\footnote{A detailed description of this kind of construction can be found in
Chapter VII.3 of \citep{Simpson}.} Since there are many such relations that collapse to a particular
hereditarily countable set, the sets are not in bijection with their
representatives. This is addressed by a quotient interpretation that
defines a natural notion of identity on these relations. The required
definable isomorphisms are then given by: sending sets in a model
of $ZFC_{count}$ to those reals that collapse to code them; and sending
numbers and sets in models of $SOA$ to those reals coding sets that
are appropriately isomorphic to them.\footnote{Note that these isomorphisms are only ``functions'' relative to
the defined notion of identity in the quotient interpretation. In
particular, two reals are deemed identical if they collapse to the
same set. } Given that a quotient is required, this gives us some reason to doubt
that these theories are definitionally equivalent. \citeauthor{EnayatLelykFOCat}
have confirmed this intuition.
\begin{thm}
\citep{EnayatLelykFOCat} $ZFC_{count}$ is not definitionally equivalent
with $SOA$, provided that there is an inaccessible cardinal or there
is an $\omega$-model of the theory extending $ZFC$ by an inaccessible
cardinal.\label{thm:E=000026L}\footnote{We note that the statement of this theorem in \citep{EnayatLelykFOCat}
makes use of the weaker assumption that $ZFC$ plus an inaccessible
cardinal is consistent. However, the proof given there seems to naturally
work by using one of the assumptions above. If we collapse an inaccessible
cardinal, then the resultant version of $H(\omega_{1})$ gives us
the model of $ZFC_{count}$ we want and the result can then be pulled
back to the ground universe. But if we merely start with a model $\mathcal{M}$
of $ZFC$ plus an inaccessible and collapse, then the argument of
the proof just shows that $\mathcal{M}$ thinks $SOA$ and $ZFC_{count}$
aren't definitionally equivalent. We have no reason to think $\mathcal{M}$
is correct about this unless $\mathcal{M}$ is, say, an $\omega$-model.
It is not so easy to say which of these assumptions is preferable.
While the first assumption has lower consistency strength than the
second, the second -- unlike the first -- can be accommodated by
theories with less ontological overheads than $ZFC$.}
\end{thm}

Unlike the toy example above, this provides a concrete example of
a pair of well-understood and very commonly used theories that are
bi-interpretable but not definitionally equivalent. We shall now prove
this claim without the inaccessible cardinal. Moreover, we'll do this
twice. First, with a simple $ZFC$ proof; and second, with an indirect
proof in $PA$. The first proof exploits the fact that certain models
of $SOA$ cannot define well-orderings of length $\geq\omega_{1}$.
The second proof exploits the fact that certain models of $ZFC_{count}$
cannot define linear orders of their domain.

\section{Simple proof in $ZFC$}

We start by recalling two well-known facts about forcing and use them
to prove a lemma from which our main claim quickly follows.
\begin{fact}
\citep{LevyDefSTI} Let $\mathbb{P}$ be weakly homogeneous and $G$
be $\mathbb{P}$-generic over $V$. Suppose $A\subseteq V$ and $A$
is definable in $V[G]$ by a formula using parameters from $V$.\footnote{See the beginning of Section 2 in \citep{WDRHoD} for a nice sketch
of the idea behind the proof of this. If such a set were definable
in $V[G]$, then the homogeneity of $\mathbb{P}$ ensures that it
can already be defined in $V$ using the forcing relation.} Then $A\in V$.\label{lem:OD}
\end{fact}

\begin{fact}
(Laver-Woodin)\footnote{See \citep{ReitzTGA} or \citep{WoodinGM} for a proof of this.}
If $V$ is a generic extension of some inner model $W$, there is
a formula defining $W$ in $V$ using a parameter from $W$. More
specifically, there is a formula $\varphi(x,y)$ and $r\in W$ such
that for all $x$\label{lem:(Laver-Woodin)}
\[
x\in W\ \Leftrightarrow\ \varphi(x,r)^{V}.
\]
\end{fact}

The following definition is the focus of our lemma below.
\begin{defn}
Let $\Gamma\subseteq\mathcal{P}(\mathbb{R})$ be a point-class. We
say that $\Gamma$ has the \emph{short well-ordering property }if
every well-ordering $R\in\Gamma$ has order type $<\omega_{1}$. 
\end{defn}

The following lemma provides the combinatorial content for our theorem.
\begin{lem}
There is a generic extension $V[G]$ of $V$ that thinks $OD\cap\mathcal{P}(\mathbb{R})$
has the short well-ordering property.\label{thm: lightPSP-SWOP}
\end{lem}

\begin{proof}
We assume $CH$ holds in $V$; if necessary, just collapse the ordinals
below $2^{\aleph_{0}}$. Now let $G$ be $Col(\omega,\{\omega_{1}\}$)-generic\footnote{Here, I'm following the definition given in Chapter 10 \citep{kanamori2003higher}.
Thus, $Col(\omega,S)$ for $S\subseteq Ord$ is intended to collapse
every ordinal in $S$ to be countable in a generic extension. It is
worth noting this since some authors will write $Col(\omega,\omega_{1})$
to denote what we are calling $Col(\omega,\{\omega_{1}\})$.} over $V$ and work in $V[G]$. Suppose toward a contradiction that
$R$ is an $OD$ -well-ordering of reals that has order-type $\geq\omega_{1}$.
Then since we've collapsed $V$'s continuum, we see that $field(R)\cap(V[G]\backslash V)\neq\emptyset$.
And since $R$ is a well-ordering, we may define the $R$-least element
$x$ of $V[G]\backslash V$. This definition only makes use of the
ordinals used in the definition of $R$ and a parameter from $V$
given by Fact \ref{lem:(Laver-Woodin)} that allows us to define $V$
in $V[G]$. Fact \ref{lem:OD}, then tells us that $x\in V$, which
is a contradiction.
\end{proof}
This next lemma shows that the combinatorial lemma is enough for a
goal.
\begin{lem}
Suppose every ordinal definable well-ordering of a set of reals is
shorter than $\omega_{1}$. Then $SOA$ is not definitionally equivalent
with $ZFC_{count}$.\label{lem:=0000ACeq}
\end{lem}

\begin{proof}
Let $S$ be the first theory and let $T$ be the second. Suppose toward
a contradiction that we have interpretations
\[
t:mod(T)\leftrightarrow mod(S):s
\]
witnessing their definitional equivalence. Note that $M=\langle H_{\omega_{1}},\in\rangle\models T$.
Then we see that $t(M)=\langle N_{0},N_{1},\sigma\rangle\models S$.\footnote{Here $N_{0}$ is the number domain, $N_{1}$ is the set domain, and
$\sigma$ is the interpretation of the non-logical vocabulary.} Now if $s\circ t(M)=M$, we must be able to define a well-ordering
of length $\omega_{1}$ in $t(M)$. We show that this is impossible. 

Note first that the naturals $N_{0}$ of $t(M)$ must be well-founded
since $\omega=\omega^{s\circ t(M)}$ is definable in $t(M)$ and if
its naturals were ill-founded, then $t(M)$ couldn't define a well-ordering
of type $\omega$. This means that we can collapse $t(M)$ to form
a model $N^{*}=\langle\omega,N_{1}^{*},\sigma^{*}\rangle\cong t(M)$
where $N_{1}^{*}\subseteq\mathcal{P}(\omega)$. Now if $t(M)$ can
define a well-ordering of length $\omega_{1}$, then so can $N^{*}$.
Moreover, since $N_{1}^{*}\subseteq\mathcal{P}(\omega)$ we may assume
that any such well-ordering defined over $N^{*}$ is a well-ordering
of a set of reals. But then since $N_{1}^{*}$ is definable from $H_{\omega_{1}}$
and $t$, we see that such a well-ordering is definable in the parameter,
$\omega_{1}$, and thus, ordinal definable. This contradicts our initial
assumption. 
\end{proof}
\begin{rem*}
The proof above also shows that $ZFC_{count}$ is not a retract (as
in ``half'' of a bi-interpretation) of $SOA$ by one-dimensional
interpretations. It can also be shown that $ZFC_{count}$ is solid\footnote{See \citep{EnayatVisTheme} for the definition of solidity and a proof
that this theory is solid.} and from this it can be shown that the model $t(M)$ in the proof
above is actually the standard model of second order arithmetic; i.e.,
$N_{1}^{*}=\mathcal{P}(\omega)$. 
\end{rem*}
Finally, we put the two lemmas together to get the desired result.
\begin{thm}
$SOA$ is not definitionally equivalent with $ZFC_{count}$.\label{thm:1st}
\end{thm}

\begin{proof}
Use Proposition \ref{thm: lightPSP-SWOP} to move to a generic extension
$V[G]$ of the universe in which every ordinal definable well-ordering
of a set of reals is shorter than $\omega_{1}$. Then Lemma \ref{lem:=0000ACeq},
tells us that the stated theories are not definitionally equivalent
in $V[G]$. The statement that these theories are definitionally equivalent
is arithmetic,\footnote{In particular, definitional equivalence can be articulated as a $\Sigma_{3}^{0}$
statement: see Fact \ref{fact:DefEq} below. To establish this, observe
that $T$ and $S$ are definitionally equivalent if \emph{there are}
\emph{natural} \emph{numbers} coding computable translations $t$
and $s$ such that: \emph{every natural number }coding a sentence
$\varphi$ such that either \emph{there is no natural number }coding
a proof witnessing $S\vdash\varphi$ or\emph{ there is a natural number}
coding a proof witnessing that $T\vdash t(\varphi)$; and \emph{every
natural number }coding a formula $\psi(\bar{x})$ in the language
of $T$ is such that \emph{there is a natural number }coding a proof
that $T\vdash\forall\bar{x}(\psi(\bar{x})\leftrightarrow t\circ s(\psi)(\bar{x}))$;
and a pair of similar clauses for sentences $\psi\in T$ and formulae
$\varphi(\bar{y})$ in the language of $S$.\label{fn:DefEq}} thus, if it is true in $V[G]$ it is also true in $V$.
\end{proof}

\section{Proof in $PA$}

As in the previous section, we start by proving a more general lemma
from which the main claim follows.\footnote{For a related result establishing that there are models of $ZFC$
with no definable global linear ordering with Cohen forcing, see Theorem
3.1 in \citep{AliEnayatLeib}.}
\begin{lem}
If we add a Cohen real to the universe $V$, then there is no ordinal
definable relation $S$ on $\mathcal{P}(\mathbb{R})$ that is connected
and asymmetric.\label{lem:}
\end{lem}

\begin{proof}
We are essentially using a simplification of the proof that the axiom
of choice fails in the second Cohen model as delivered in Theorem
5.19 of \citep{jechAC}. The plan is to describe, in a generic extension,
a pair $P$ of sets of reals that has no ordinal definable element.
To see that this suffices, suppose there was an ordinal definable
relation $S$ on $\mathcal{P}(\mathbb{R})$ that is connected and
asymmetric. Then we may obtain an ordinal definable element of the
pair $P$ by taking the $S$-least element of the $P$. 

We start by defining a forcing $\mathbb{P}=\langle P,\leq\rangle$
where $P$ consists of partial functions 
\[
p:(2\times\omega)\times\omega\rightharpoondown2
\]
ordered by reverse inclusion. This is intended to deliver us a pair
of sets each containing infinitely many Cohen reals. But do note that,
in the codes, this is a minor variation of the usual forcing to add
a Cohen real. Let $G$ be $\mathbb{P}$-generic over $V$. We then
define some useful names and the objects they denote in $V[G]$. For
$n,i,j\in\omega$ and $e\in2$, we let
\begin{itemize}
\item $\dot{x}_{e,i}=\{\langle\check{j},p\rangle\ |\ p(n,e,i,j)=1\}$; 
\item $x_{e,i}=\{j\in\omega\ |\ \exists p\in G\ p(n,e,i,j)=1\}$ be a real;
\item $\dot{X}_{e}=\{\langle\dot{x}_{n,e,i},1\rangle\ |\ i\in\omega\}$; 
\item $X_{e}=\{x_{n,e,i}\ |\ i\in\omega\}$ be a countable set of reals; 
\item $\dot{P}=\{\langle\dot{X}_{n,0},1\rangle,\langle\dot{X}_{n,1},1\rangle\}$;
and
\item $P=\{X_{n,0},X_{n,1}\}$ be a pair of countable sets of reals
\end{itemize}
This gives us the set $P$ that we want. $P$ contains exactly two
sets $X_{0}$ and $X_{1}$. $X_{0}$ contains a set of Cohen reals
$x_{0,i}$ for all $i\in\omega$; and similarly, for $X_{n,1}$. Note
that a simple density argument reveals that $\Vdash x_{e,i}\neq x_{e^{*},i^{*}}$
whenever $\langle e,i\rangle\neq\langle e^{*},i^{*}\rangle$. Thus,
$\Vdash\dot{X}_{0}\neq\dot{X}_{1}$. Next observe that any permutation
of $2\times\omega$ delivers an automorphism of $\mathbb{P}$ that
can be extended to a map from $V^{\mathbb{P}}$ to itself in a natural
way. Moreover, for any such automorphism, we have 
\[
p\Vdash\varphi(\dot{y}_{0},...,\dot{y}_{n})\Leftrightarrow\pi(p)\Vdash\varphi(\pi\dot{y}_{0},...,\pi\dot{y}_{n}).
\]
We now claim that there is no ordinal definable set denoted by a term
$t=t(\bar{\alpha})$ such that $t\in P$ . To see this, suppose toward
a contradiction that there is some $p_{0}\in G$ that forces that
there is such a $t$. Note that since $t$ is ordinal definable, it
is not affected by automorphisms of $\mathbb{P}$. We may then fix
some $p\leq p_{0}$ with $p\in G$ that decides the value of $t$.
For definiteness, suppose that $t=X_{0}$, so we have 
\[
p\Vdash t=\dot{X}_{0}.
\]
Now it can then be seen that there is an automorphism $\pi$ of $\mathbb{P}$
such that:
\begin{itemize}
\item $\pi(p)\not\perp p$; 
\item $\pi\dot{P}=\dot{P}$; and 
\item $\pi\dot{X}_{0}=\dot{X}_{1}$. 
\end{itemize}
We just describe the underlying permutation and leave the proof of
these facts to the reader.\footnote{A definition of a very similar permutation $\pi$ can be found at
the end of the proof of Lemma 5.19 at the top of page 71 in \citep{jechAC}.} First, fix a sufficiently large $k\in\omega$ that for all $i\geq k$,
$p$ cannot decide $x_{0,i}\in X_{0}$ or $x_{1,i}\in X_{1}$. We
continue by informally describing $\pi$ by its behavior in the generic
extension. First, we swap the interval $[0,k)$ of Cohen reals associated
with $X_{0}$ with the interval $[k,2k)$ associated with $X_{1}$.
Then we swap the interval $[0,k)$ of Cohen reals associated with
$X_{1}$ with the interval associated with $[k,2k)$ in $X_{0,0}$.
Above $2k$, we just swap the Cohen reals associated with $X_{0}$
with those of $X_{1}$. Much more formally, $\pi:2\times\omega\to2\times\omega$
is such that for all $e\in2$ and $i\in\omega$
\[
\pi(e,i)=\begin{cases}
\langle e-1,i+k\rangle & \text{if }i\in[0,k)\\
\langle e-1,i-k\rangle & \text{if }i\in[k,2k)\\
\langle e-1,i\rangle & \text{if }i\in[k,\omega)
\end{cases}
\]

Recalling that $t$ is unaffected by $\pi$, we see that 
\[
\pi(p)\Vdash t=\pi\dot{X}_{0}
\]
and so 
\[
p\cup\pi(p)\Vdash t=\dot{X}_{1}
\]
which is a contradiction, since $\Vdash\dot{X}_{0}\neq\dot{X}_{1}$.
\end{proof}
The Lemma above is demonstrated using $ZFC$. However, it is not difficult
to see that the proof can be adapted to the context of $ZFC_{count}$
where we use ordinary definability rather than ordinal definability.
Moreover, the proof works to show that there is no such definable
relation on any domain (including the universe itself) that extends
the set $A$ delivered in the proof. 
\begin{cor}
($ZFC_{count}$) If we add a Cohen real to the universe, then there
is no definable relation $S$ on the extended universe that is connected
and asymmetric.\label{cor:CohenZFC-count} More precisely, for all
formula $\varphi_{S}(x,y)$ of $\mathcal{L}_{\in}$, we have 
\[
\Vdash_{Add(\omega,1)}\exists x\exists y(\varphi_{S}(x,y)\leftrightarrow\varphi_{S}(y,x)).
\]
\end{cor}

Before, we prove the result in $PA$, let us first give a quick proof
in $ZFC_{count}$ that may give a clearer picture of the underlying
idea. 
\begin{thm}
($ZFC_{count}$) $SOA$ is not definitionally equivalent with $ZFC_{count}$,
if one of these theories is consistent.\label{thm:main}
\end{thm}

Before we give the proof, let us first discuss the statement of this
theorem and how it differs from Theorems \ref{thm:E=000026L} and
\ref{thm:1st}. The first thing to note is that Theorem \ref{thm:E=000026L}
makes use of a consistency assumption that cannot be proved in $ZFC$.
Theorem \ref{thm:1st} arguably improves this by removing that assumption
and just proving the result in $ZFC$. Theorem \ref{thm:main}, however,
also uses a consistency statement. One might think that this has taken
us a step backwards, but this would be misguided. Here, we require
the assumption since $ZFC_{count}$ cannot prove the consistency of
either $SOA$ or $ZFC_{count}$, although it can prove their equiconsistency.
Without the ability to prove the consistency of one of these theories,
it will not be possible to deliver a model witnessing the failure
definitional equivalence. Moreover, if one (and thus, both) of them
are inconsistent they will be vacuously definitionally equivalent
since they have no models. In contrast, the background assumptions
of both Theorems \ref{thm:E=000026L} and \ref{thm:1st} are sufficient
to prove the consistency of both $SOA$ and $ZFC_{count}$. Thus,
Theorem \ref{thm:1st} has also been improved by using a background
theory that is insufficient to prove the consistency of its target
theories.Of course, other theories weaker than $ZFC_{count}$, like
$KP$, will also be able to prove the theorem above. Our final result
below is proved in the arithmetic theory $PA$, which seems fitting
since definitional equivalence can be articulated as an arithmetic
statement. 
\begin{proof}
 We work informally in $ZFC_{count}$. Suppose toward a contradiction
that we have interpretations
\[
t:mod(ZFC_{count})\leftrightarrow mod(SOA):s
\]
witnessing that $SOA$ and $ZFC^{-}$ are definitionally equivalent.
Let $\mathcal{M}$ be a model of $ZFC_{count}$ and without loss of
generality suppose that it is countable and satisfies $V=L$. Now
let $c$ be a Cohen real over $\mathcal{M}$. Then $\mathcal{M}[c]$
satisfies $ZFC_{count}$ and the statement that its universe is constructed
from $c$; i.e., $V=L[c]$.\footnote{Note that since there is no guarantee that $\mathcal{M}$ is well-founded,
we cannot define $\mathcal{M}[G]$ using the standard $Val$ function
that is familiar from \citep{KunenST} or \citep{ShoenUnramForc}.
Rather, we define membership and identity in the model using the forcing
relation. So, for example, given $\mathbb{P}$-names $\dot{x}$ and
$\dot{y}$, we let $\dot{x}\in^{\mathcal{M}[G[}\dot{y}$ iff there
is some $p\in G$ such that $p\Vdash\dot{x}\in\dot{y}$. See \citep{CorFNWF}
or \citep{MeadMadPGMV} for more details.} Using Lemma \ref{cor:CohenZFC-count} in $\mathcal{M}[c]$, we see
that $\mathcal{M}[c]$ cannot define a linear ordering of its domain.
On the other hand, $t(\mathcal{M}[c])$, as a model of $SOA$, can
easily define a linear order of its entire domain using, say, the
lexicographic ordering of $2^{\omega}$. But since $ZFC_{count}$
and $SOA$ are definitionally equivalent, we see that $\mathcal{M}[c]$
and $t(\mathcal{M}[c])$ share the same domain and thus, any relation
definable over $t(\mathcal{M}[c])$ is also definable $\mathcal{M}[c]$.
This is a contradiction. 
\end{proof}
We're almost ready for the final result, but it will be helpful to
first give this alternative, syntactic characterization of definitional
equivalence that is amenable to use in theories of arithmetic.
\begin{fact}
Let $T_{0}$ and $T_{1}$ be theories articulated in $\mathcal{L}_{0}$
and $\mathcal{L}_{1}$ respectively. Then $T_{0}$ and $T_{1}$ are
definitionally equivalent, if there are translations $t_{0}:\mathcal{L}_{1}\to\mathcal{L}_{0}$
and $t:\mathcal{L}_{0}\to\mathcal{L}_{1}$ giving rise to direct interpretations
such that for $i\in\{0,1\}$:\footnote{See \citep{Visser2006} for more discussion of this and similar results.}\label{fact:DefEq}
\begin{enumerate}
\item For all sentences $\varphi\in\mathcal{L}_{i}$, if $T_{i}\vdash\varphi$,
then $T_{i-1}\vdash t_{i-1}(\varphi)$; and
\item For all formulae $\varphi(\bar{x})\in\mathcal{L}_{i}$, $T_{i}\vdash\forall\bar{x}(\varphi(\bar{x})\leftrightarrow t_{i}\circ t_{i-1}(\varphi)(\bar{x}))$.
\end{enumerate}
\end{fact}

Finally, we are in position to demonstrate the main claim, by proving
a lemma from which the result follows trivially.
\begin{lem}
($PA$) $ZFC_{count}$ cannot interpret $SOA$ via a direct interpretation;
i.e., an interpretation that preserves identity and the domain, if
either $ZFC_{count}$ or $SOA$ is consistent. 
\end{lem}

\begin{proof}
We proceed by contraposition working informally in $PA$. Thus, we
suppose that we have a translation $t:\mathcal{L}_{SOA}\to\mathcal{L}_{\in}$
giving rise to a direct interpretation such that for all $\varphi\in\mathcal{L}_{SOA}$,
if $SOA\vdash\varphi$, then $ZFC_{count}\vdash t(\varphi)$. And
we aim to show that $ZFC_{count}$ and $SOA$ are inconsistent. First
observe that using, say, the lexicographic ordering on subsets of
$\omega$, there is formula $\psi(x,y)$ such that 
\[
SOA\vdash\forall x\forall y(\psi(x,y)\leftrightarrow\neg\psi(y,x)).
\]
Thus, we see that 
\[
ZFC_{count}\vdash\forall x\forall y(t(\psi)(x,y)\leftrightarrow\neg t(\psi)(y,x)).
\]

However, by Corollary \ref{cor:CohenZFC-count}, we know that for
all formulae $\chi(x,y)$ of $\mathcal{L}_{\in}$, $ZFC_{count}$
proves that 
\begin{equation}
\Vdash_{Add(\omega,1)}\exists x\exists y(\chi(x,y)\leftrightarrow\chi(y,x)).\label{eq:1}
\end{equation}
Moreover, by standard arguments we know that $ZFC_{count}$ proves
the following:
\begin{enumerate}
\item $\Vdash_{Add(\omega,1)}\varphi$, for all axioms $\varphi$ of $ZFC_{count}$; 
\item $\Vdash_{Add(\omega,1)}$ is closed under proof in first order logic
(i.e., if $\Gamma\vdash\varphi$ and $\Vdash_{Add(\omega,1)}\gamma$
for all $\gamma\in\Gamma$, then $\Vdash_{Add(\omega,1)}\varphi$);
and
\item $\nVdash_{Add(\omega,1)}\bot$.
\end{enumerate}
Note that each of these claims are proved in $ZFC_{count}$ not our
background theory $PA$. As such, we can pluck these results directly
from the textbooks.\footnote{The classic location for these results is \citep{KunenST}. However,
note that this is not the approach taken in the main text. Rather,
it is Approach 2 described in Section VII.9 on page 234 of that book.
Similar results in the context of Boolean algebras can be found in
\citep{BellBVM}, however, note that we cannot use the Boolean algebra
approach here since we are working in $ZFC_{count}$ and without powerset
we cannot form the completion of $Add(\omega,1)$. Nonetheless, the
proofs are very similar.} Using the fact that $PA$ proves internal $\Sigma_{1}^{0}$-completeness\footnote{See, for example, (BLiii) on page 17 of \citep{lindstrm2003aspects}.}
we see that $ZFC_{count}$ proves that $ZFC_{count}\vdash\forall x\forall y(t(\psi)(x,y)\leftrightarrow\neg t(\psi)(y,x))$
and thus, we may use (2) and (1) to show that that $ZFC_{count}$
proves 
\[
\Vdash_{Add(\omega,1)}\forall x\forall y(t(\psi)(x,y)\leftrightarrow\neg t(\psi)(y,x)).
\]
Moreover, using \ref{eq:1}, we see that 
\[
\Vdash_{Add(\omega,1)}\exists x\exists y(t(\psi)(x,y)\leftrightarrow t(\psi)(y,x))
\]
and so $ZFC_{count}$ proves that $\Vdash_{Add(\omega,1)}\bot$. This
implies that $ZFC_{count}$ is inconsistent. And since the equiconsistency
of $ZFC_{count}$ and $SOA$ is clearly provable in $PRA$, we see
that $SOA$ is also inconsistent as required. 
\end{proof}
Then since interpretations witnessing definitional equivalence must
be direct, we see that:
\begin{cor}
($PA$) $ZFC_{count}$ and $SOA$ are not definitionally equivalent,
if one of those theories is consistent. 
\end{cor}

Thus, we have a proof of the arithmetic claim that $ZFC_{count}$
and $SOA$ are not definitionally equivalent conducted in a standard
theory of arithmetic, $PA$. One might hope that to obtain the result
above from a weaker theory like $PRA$. However, the use of internal
$\Sigma_{1}^{0}$-completeness in the proof above seems to tell us
that, at least for this strategy, some form of induction is required
in our metatheory.\footnote{We've tried a few other, less elegant, approaches and wherever we've
looked some induction seems to be required. } As such, we leave open the question of whether $PRA$ can prove that
$ZFC_{count}$ and $SOA$ are not definitionally equivalent.

\bibliographystyle{plainnat}
\bibliography{../../../Thesis}

\end{document}